\documentclass[11pt]{article}
\usepackage[a4paper, total={17cm,25cm}]{geometry}

\usepackage{amsmath,amssymb,amsfonts}
\usepackage{color, graphics}
\usepackage{tikz,tkz-euclide}
\usetikzlibrary{arrows,arrows.meta,patterns,calc,math,bending,cd}
\usepackage{float}

\usepackage[most]{tcolorbox}

\usepackage{pst-all}
\usepackage{pst-solides3d}

\usepackage[new]{old-arrows}

\usepackage{mathrsfs}

\usepackage[alphabetic,y2k,initials]{amsrefs}

\usepackage[normalem]{ulem}
\usepackage{hyperref}
\usepackage{titlesec}

\numberwithin{equation}{section}

\newtheorem{theorem}{Theorem}[section]

\newtheorem{remark}[theorem]{Remark}

\newenvironment{proof}[1][Proof]{\noindent\textit{#1.} }{\hfill$\Box$
	
\medskip}

 \title{The Weighted Walks in Quadrant with Finite Groups: an Algebro-Geometric Approach}

\usepackage{authblk}
\author[1,3]{Vladimir Dragovi\'c}
\author[2,3]{Milena Radnovi\'c}
\affil[1]{\textsc{The University of Texas at Dallas, Department of Mathematical Sciences}}
\affil[2]{\textsc{The University of Sydney, School of Mathematics and Statistics}}
\affil[3]{\textsc{Mathematical Institute SANU, Belgrade}}
\affil[ ]{\texttt{vladimir.dragovic@utdallas.edu, milena.radnovic@sydney.edu.au}}

\date{}

\begin{document}

\maketitle

\begin{abstract}
We classify weighted small-step lattice walks in the quadrant whose associated birational group is finite. Using an algebro-geometric description of the kernel curves and Cayley-type finite-order conditions, we relate the group $G_W$ of a walk to the family of groups $\Gamma_t$ acting on its kernel curves. Together with the uniform upper bound on the order of $G_W$, this allows us to analyse all possible finite orders. We obtain explicit necessary and sufficient conditions for the group to have order $4$, $6$, $8$, or $10$, and prove that no weighted walk in the quadrant has a group of order $12$. This yields a complete classification of weighted quadrant walks with finite groups.
\end{abstract}

\medskip
\noindent\textbf{Keywords.} Weighted lattice walks; walks in the quadrant; kernel curves; finite groups; birational involutions; elliptic curves.

\medskip
\noindent\textbf{2020 Mathematics Subject Classification.} 05A15; 14H52; 14H70; 60G50; 20F55; 39B32.


\section{Introduction}
Planar lattice walks in the quarter plane have attracted considerable interest over the past twenty years; see, for example, \cite{BMM, KuRas, Ras, BKR}. However, their study has a much longer history; see, for example, \cite{Mal1}. Such walks have numerous applications in combinatorics and discrete mathematics, probability, statistics, and mathematical physics; see \cite{HS} and the references therein. One of the central questions concerns the analytic nature of the associated generating functions; see, for example, \cite{Mal, RandomWalks, FayRas1, HS} and the references therein. This question has led, in particular, to the study of the groups associated with lattice-walk models and the conditions under which these groups are finite.

Very recently, in \cite{DRrw}, the authors obtained necessary and sufficient conditions for the groups associated with the kernel curves of a weighted walk to have a prescribed finite order $2n$. The goal of the present paper is to extend the methods of \cite{DRrw} and give a complete description of all weighted walks in the quadrant for which the group $G_W$ is finite.

Our main results are logically independent of previous work, apart from an important uniform bound established in \cite{HS}, which states that the order of a finite group $G_W$ is at most $12$. They are consistent with the results formulated in \cite{KaYa} and with the planar part of \cite{EHR}. The proofs given here are self-contained. 

The paper is organised as follows. In the next section, we introduce the group of a weighted walk, recall the kernel-curve criterion from \cite{DRrw}, and establish the relation between the finiteness of $G_W$ and that of the groups acting on the kernel curves. The subsequent sections treat, in turn, the possible group orders $4$, $6$, $8$, $10$, and $12$.

\section{Kernel curves and finite-group criteria}\label{sec:preliminaries}

We consider weighted walks in the quadrant under the small-step assumption.
The inventory of such a walk $W$ is:
$$
\chi_W(x,y)=\sum_{i,j=-1}^1 p_{ij}x^iy^j,
$$
with
$$
p_{ij}\ge0,
\quad
\sum_{i,j=-1}^1p_{ij}=1.
$$
We can rewrite:
$$
\chi_W(x,y)
=
xA_1(y)+B_1(y)+\frac{C_1(y)}x
=
yA_2(x)+B_2(x)+\frac{C_2(x)}y.
$$
Following \cite{BMM}, we define the group of $W$ as $G_W=\langle \varphi_1,\varphi_2\rangle$ of birational transformations of the $(x,y)$-plane generated by the involutions:
\begin{equation}\label{eq:phi}
\varphi_1(x,y)=\left(\frac{C_1(y)}{xA_1(y)},y\right),
\quad
\varphi_2(x,y)=\left(x,\frac{C_2(x)}{yA_2(x)}\right).
\end{equation}

Let us consider a family of curves $\mathcal C_t$, parameterized by $t$:

\begin{equation}\label{eq:biqrandomt}
\mathcal C_t\ :\ Q_t(x,y)=xy\left(\sum_{i,j=-1}^1 p_{ij}x^iy^j-\frac{1}{t}\right)=0,
\end{equation}
where $Q_t(x,y)$ is the so-called kernel of the walk and the curves $\mathcal C_t$ are the kernel curves of the walk (see e.g.~\cite{HS}).
Notice that horizontal and vertical involutions on each of the kernel curves are given by \eqref{eq:phi}, and denote by $\Gamma_t$ the group generated by those involutions on $\mathcal{C}_t$, following Malyshev, see e.g.~\cite{Mal, RandomWalks}.

\begin{remark}
Groups $G_W$ and $\Gamma_t$ are dihedral groups, since they are generated by a pair of involutions.
Notice that a dihedral group generated by involutions $\iota_1$, $\iota_2$ is finite if and only if $\iota_1\circ\iota_2$ is of finite order $n$, and then the order of the group equals $2n$. 

\end{remark}

We will use the following uniform bound from \cite{HS}.

\begin{theorem}[\cite{HS}]\label{th:uniformbound} 
If a group $G_W$ is finite, then its order is not greater than $12$.
\end{theorem}

For the proof, see Remark 5.1 in \cite{HS}.

We denote the matrix of $\mathcal{C}_t$ as follows:
\begin{equation}\label{eq:P}
	P_t=\begin{pmatrix}
		p_{11} & p_{10} &  p_{1,-1}\\
		p_{01} &  p_{00}-1/t &  p_{0,-1}\\
		p_{-1,1} &  p_{-1,0} &  p_{-1,-1}
	\end{pmatrix}.
\end{equation}

The following theorem follows directly from Theorem 3.3 of \cite{DRrw}.

\begin{theorem}\label{th:cayley}
The restriction of $(\varphi_1\circ\varphi_2)^n$ to $\mathcal{C}_t$ is the identity on that curve if and only if:
\begin{itemize}
	\item [(i)] $n=2$ and $\det P_t=0$;
	\item [(ii)] $n=2k+1$, $k\ge1$, and
	$$
	\det\begin{pmatrix}
		C_2 & C_3 & \dots & C_{k+1}\\
		C_3 & C_4 & \dots & C_{k+2}\\
		&&\dots& \\
		C_{k+1} & C_{k+2} &\dots & C_{2k}
	\end{pmatrix}
	=0;
	$$
	\item [(iii)] $n=2k$, $k\ge2$, and
	$$
	\det\begin{pmatrix}
		C_3 & C_4 & \dots & C_{k+1}\\
		C_4 & C_5 & \dots & C_{k+2}\\
		&&\dots& \\
		C_{k+1} & C_{k+2} &\dots & C_{2k-1}
	\end{pmatrix}
	=0.
	$$
\end{itemize}
Here, the entries $C_k$ of the matrices are the coefficients in the following Taylor expansion:
\begin{equation}\label{eq:taylor}
	\sqrt{4x^3- Dx+E}
	=
	C_0+C_1(x-X_0)+C_2(x-X_0)^2+C_3(x-X_0)^3+\dots,
\end{equation}
where
$$
X_0
=\frac{(p_{00}-1/t)^2-4p_{01}p_{0,-1}-4p_{10}p_{-1,0}+8p_{-1,1}p_{1,-1}+8p_{11}p_{-1,-1}}{12},
$$
and
\begin{align*}
	D=\, &
	\frac{1}{12} \left((p_{00}-1/t)^2-4 p_{-1,-1} p_{22}-4 p_{-1,0} p_{10}-4 p_{-1,1} p_{1,-1}+2 p_{0,-1} p_{01}\right)^2
	\\&
	-( p_{0,-1} (p_{00}-1/t)-2 p_{-1,-1} p_{10}-2 p_{-1,0} p_{1,-1})
	( (p_{00}-1/t) p_{01}-2 p_{-1,0} p_{22}-2 p_{-1,1} p_{10})
	\\&
	+\left(p_{0,-1}^2-4 p_{-1,-1} p_{1,-1}\right) \left(p_{01}^2-4 p_{-1,1} p_{22}\right),\\
	E=\, &
	-\frac{1}{6} \left(p_{0,-1}^2-4 p_{-1,-1} p_{1,-1}\right)
	\left(p_{01}^2-4 p_{-1,1} p_{22}\right)\times
	\\&
	\quad\times\left((p_{00}-1/t)^2-4 p_{-1,-1} p_{22}-4 p_{-1,0} p_{10}-4 p_{-1,1} p_{1,-1}+2 p_{0,-1} p_{01}\right)
	\\&
	+\frac{1}{4} \left(p_{0,-1}^2-4 p_{-1,-1} p_{1,-1}\right)
	((p_{00}-1/t) p_{01}-2 p_{-1,0} p_{22}-2 p_{-1,1} p_{10})^2
	\\&
	+\frac{1}{216} \left((p_{00}-1/t)^2-4 p_{-1,-1} p_{22}-4 p_{-1,0} p_{10}-4 p_{-1,1} p_{1,-1}+2 p_{0,-1} p_{01}\right)^3
	\\&
	-\frac{1}{12} (p_{0,-1} (p_{00}-1/t)-2 p_{-1,-1} p_{10}-2 p_{-1,0} p_{1,-1})
	((p_{00}-1/t) p_{01}-2 p_{-1,0} p_{22}-2 p_{-1,1} p_{10})
	\times
	\\&\ \ \ \times
	\left((p_{00}-1/t)^2-4 p_{-1,-1} p_{22}-4 p_{-1,0} p_{10}-4 p_{-1,1} p_{1,-1}+2 p_{0,-1} p_{01}\right)
	\\&
	+\frac{1}{4} \left(p_{01}^2-4 p_{-1,1} p_{22}\right)
	(p_{0,-1} (p_{00}-1/t)-2p_{-1,-1} p_{10}-2 p_{-1,0} p_{1,-1})^2.
\end{align*}
\end{theorem}

Similar conditions appeared before \cite{Cayley1853} in a different context, see also \cite{DR2011knjiga} and references therein. See also \cite{DuistermaatBOOK} and references therein.

\begin{theorem}\label{th:finite-group}
The group $G_W$ is of finite order $2n$ if and only if there are infinitely many values of $t$ for which $\Gamma_t$ is of the order $2n$.
Moreover, then all groups $\Gamma_t$ are finite and the order of each of them  divides $2n$.
\end{theorem}
\begin{proof}
First, suppose that $|\Gamma_t|=2n$ for infinitely many values of $t$. 
Then the restriction of $\varphi_1\circ\varphi_2$ to $\mathcal{C}_t$ is of order $n$ for infinitely many $t$, so the corresponding condition from Theorem \ref{th:cayley}, which is polynomial in $t$, must be identically equal to zero, i.e.~it is satisfied for each $t$.
Thus $\varphi_1\circ\varphi_2$ is of order $n$, so $|G_W|=2n$.
We can also conclude that $\Gamma_t$ is finite and of order which divides $2n$, for each $t$.
Notice that it means that there must be only finitely many values of $t$ for which $|\Gamma_t|\neq2n$.

Now, suppose that $|G_W|=2n$.
That immediately implies that all groups $\Gamma_t$ are also finite and that order of each of them divides $2n$. 
If infinitely many of those groups would have order smaller than $2n$, the first part of the proof would imply $|G_W|<2n$.
Thus, almost all groups have order $2n$ and the remaining finitely many of them have orders that divide that number.
\end{proof}

We denote by $K_n(t)=0$ the condition that the group associated with the kernel curve $Q_t=0$ has order $2n$.

\section{Groups of order 4}

In this section, we analyse the case of weighted walks whenever the group $G_W$ is of order $4$. We have:

\begin{theorem}\label{th:order4}
The group $G_W$ is of order $4$ if and only if the following two matrices are singular:
$$
P_{00}
=
\begin{pmatrix}
 p_{11} &  p_{1,-1}\\
 p_{-1,1}  &  p_{-1,-1}
\end{pmatrix},
\quad
P_{\infty}
=
\begin{pmatrix}
 p_{11} & p_{10} &  p_{1,-1}\\
		 p_{01} &  p_{00} &  p_{0,-1}\\
		 p_{-1,1} &  p_{-1,0} &  p_{-1,-1}
\end{pmatrix},
$$
i.e.
$$
\det
\begin{pmatrix}
 p_{11} &  p_{1,-1}\\
 p_{-1,1}  &  p_{-1,-1}
\end{pmatrix}
=
\det
\begin{pmatrix}
 p_{11} & p_{10} &  p_{1,-1}\\
		 p_{01} &  p_{00} &  p_{0,-1}\\
		 p_{-1,1} &  p_{-1,0} &  p_{-1,-1}
\end{pmatrix}
=0.
$$
More explicitly, a weighted walk has the group $G_W$ of order four if and only if the walk is of the form:
$$
P_{\infty}
=\begin{pmatrix}
	\mu\rho & \alpha\rho  &   \nu\rho \\
	p_{01} &  p_{00} &  p_{0,-1}\\
	\mu \sigma &  \alpha \sigma  &  \nu \sigma
\end{pmatrix}
\quad\text{or}\quad
P_{\infty}
=\begin{pmatrix}
	\mu\rho & p_{10} &  \nu\rho\\
	\mu \alpha  &  p_{00} & \nu \alpha \\
	\mu\sigma &  p_{-1,0} &  \nu\sigma
\end{pmatrix},
$$
for some parameters $\alpha$, $\mu$, $\nu$, $\rho$, $\sigma$.
\end{theorem}
\begin{proof}
According to Theorems \ref{th:finite-group} and \ref{th:cayley} (case (i)), that case takes place if and only if $\det P_t=0$ for all $t$. This is equivalent to:
$$
-\frac1t\det
\begin{pmatrix}
 p_{11} &  p_{1,-1}\\
 p_{-1,1}  &  p_{-1,-1}
\end{pmatrix}
+
\det
\begin{pmatrix}
 p_{11} & p_{10} &  p_{1,-1}\\
		 p_{01} &  p_{00} &  p_{0,-1}\\
		 p_{-1,1} &  p_{-1,0} &  p_{-1,-1}
\end{pmatrix}
=0
\quad\text{for each }t.
$$

The condition $\det P_{00}=0$ implies that there are parameters $\rho$, $\sigma$, $\mu$, $\nu$ such that

$$
\begin{pmatrix}
	p_{11}&p_{1,-1}\\
	p_{-1,1}&p_{-1,-1}
\end{pmatrix}
=
\begin{pmatrix}
\mu\rho   &   \nu\rho \\
	\mu \sigma  &  \nu \sigma
\end{pmatrix}.
$$
Thus:
$$
\det P_{\infty}
=
-(\rho p_{-1,0}-\sigma p_{10})
(\mu p_{0,-1}-\nu p_{01})=0,
$$
which implies the statement.
\end{proof}

\section{Groups of order 6}

We classify here the weighted walks in quadrant with groups $G_W$ of order $6$. Thus, we describe the probability matrices $P_{\infty}$ for which the associated group of the weighted walk has
order \(6\) for every value of \(t\).

\begin{theorem}\label{th:order6} The weighted walks in quadrant have the groups $G_W$ of order $6$ if and only if ,
up to the symmetries, the probability matrices belong to one of the
	following three families:
	\[
	\begin{pmatrix}
		0 & p_{10} & 0\\
		0 & p_{00} & p_{0,-1}\\
		p_{-1,1} & 0 & 0
	\end{pmatrix},
	\qquad
	p_{10}p_{0,-1}p_{-1,1}\neq0,
	\tag{I}
	\]
	\[
	\begin{pmatrix}
		0 & p_{10} & p_{1,-1}\\
		0 & p_{00} & p_{0,-1}\\
		p_{-1,1} &
		\dfrac{2p_{1,-1}p_{-1,1}}{p_{10}} &
		\dfrac{p_{1,-1}^{\,2}p_{-1,1}}{p_{10}^{\,2}}
	\end{pmatrix},
	\qquad
	p_{10}p_{1,-1}p_{-1,1}\neq0,
	\tag{II}
	\]
	or
	\[
	\begin{pmatrix}
		0 &
		\dfrac{p_{1,-1}p_{-1,1}}{p_{-1,0}} &
		p_{1,-1}\\
		\dfrac{p_{1,-1}p_{-1,1}}{p_{0,-1}} &
		p_{00} &
		p_{0,-1}\\
		p_{-1,1} &
		p_{-1,0} &
		0
	\end{pmatrix},
	\qquad
	p_{1,-1}p_{-1,1}p_{0,-1}p_{-1,0}\neq0.
	\tag{III}
	\]
\end{theorem}
\begin{proof}
According to Theorems \ref{th:finite-group} and \ref{th:cayley}, the group $G_W$ is of order $2n=6$ if and only if $C_2(t)=0$ for all $t$.

Calculation gives that:
$$
C_2(t)=\frac2t\cdot\frac{
	k_0+k_1 t+k_2 t^2+k_3 t^3+k_4 t^4
}
{(\det P_{00}
	-
	t\det P_{\infty})^3
}
,
$$
where $k_0$, \dots, $k_4$ are polynomial in $p_{ij}$.
Hence, the necessary and sufficient condition for groups $G_W$ to be of order $6$ is:
\[
	k_0=k_1=k_2=k_3=k_4=0.
	\]
	
	Since
	\[
	k_0=-p_{11}p_{1,-1}p_{-1,1}p_{-1,-1},
	\]
	at least one corner coefficient must vanish. By symmetry, we may suppose
	that
	\[
	p_{11}=0.
	\]
With that assumption, we have
	\[
	k_1=-p_{01}p_{10}p_{1,-1}p_{-1,1}p_{-1,-1}.
	\]
	Thus, after applying symmetry if necessary, it remains to consider the
	following cases:
	\[
	p_{01}=0,\qquad p_{1,-1}=0,\qquad\text{or}\qquad p_{-1,-1}=0.
	\]
	
	\medskip
	\noindent\textbf{Case 1: \(p_{11}=p_{01}=0\).}
	
	Here
	\[
	k_2
	=
	p_{1,-1}p_{-1,1}^{\,2}
	\bigl(p_{1,-1}^{\,2}p_{-1,1}-p_{10}^{\,2}p_{-1,-1}\bigr).
	\]
	Moreover,
	\[
	\det P_{00}=-p_{1,-1}p_{-1,1},
	\]
	and
	\[
	\det P_{\infty}
	=
	p_{-1,1}
	\det
	\begin{pmatrix}
		p_{10} & p_{1,-1}\\
		p_{00} & p_{0,-1}
	\end{pmatrix}.
	\]
	Therefore \(p_{-1,1}\neq0\). We obtain two subcases.
	
	\smallskip
	\noindent\emph{Case 1A: \(p_{1,-1}=0\).}
	
	Since \(\det P_{00}=0\), non-degeneracy requires
	\[
	\det P_{\infty}=p_{-1,1}p_{10}p_{0,-1}\neq0.
	\]
	Furthermore,
	\[
	k_3=-p_{0,-1}p_{10}^{\,3}p_{-1,1}^{\,2}p_{-1,-1},
	\]
	so \(p_{-1,-1}=0\), and then
	\[
	k_4=-p_{0,-1}^{\,2}p_{10}^{\,3}p_{-1,0}p_{-1,1}^{\,2},
	\]
	so \(p_{-1,0}=0\). This gives family \emph{(I)}.
	
	\smallskip
	\noindent\emph{Case 1B: \(p_{1,-1}\neq0\).}
	
	Then
	\[
	p_{1,-1}^{\,2}p_{-1,1}=p_{10}^{\,2}p_{-1,-1}.
	\]
	In particular, that implies \(p_{10}p_{-1,-1}\neq0\), and hence
	\[
	p_{-1,-1}
	=
	\frac{p_{1,-1}^{\,2}p_{-1,1}}{p_{10}^{\,2}}.
	\]
	Substitution into \(k_3\) gives
	\[
	k_3=
	p_{0,-1}p_{10}p_{1,-1}p_{-1,1}^{\,2}
	\bigl(2p_{1,-1}p_{-1,1}-p_{10}p_{-1,0}\bigr),
	\]
thus \(p_{0,-1}=0\) or $2p_{1,-1}p_{-1,1}=p_{10}p_{-1,0}$.
	If \(p_{0,-1}=0\), then the expression for \(k_4\):
$$
k_4
=
-p_{1,-1}^2 p_{-1,1}^2 (p_{10} p_{-1,0}-2 p_{1,-1} p_{-1,1})^2,
$$
implies again $2p_{1,-1}p_{-1,1}=p_{10}p_{-1,0}$, thus that relation must hold, which yields family \emph{(II)}.
	
	\medskip
	\noindent\textbf{Case 2: \(p_{11}=p_{1,-1}=0\).}
We can assume here $p_{10}p_{01}p_{0,-1}\neq0$, otherwise the discussion can be reduced to Case 1.
	
	In this case,
	\[
	k_2=-p_{01}p_{0,-1}p_{10}^{\,2}p_{-1,1}p_{-1,-1}.
	\]

The possible cases are \(p_{-1,1}=0\) and \(p_{-1,-1}=0\), which are symmetric to each other, so we can consider only \(p_{-1,1}=0\).
Note that we can also assume $p_{-1,0}\neq0$ because of symmetry with Case 1 and $p_{-1,-1}\neq0$ because of the non-degeneracy of $P_{\infty}$.
Then:
$$
k_3=-p_{01}^2 p_{0,-1} p_{10}^2 p_{-1,0} p_{-1,-1}\neq0,
$$
hence there are no further solutions in this case.
	
	\medskip
	\noindent\textbf{Case 3:
		\(p_{11}=p_{-1,-1}=0.\)}
We can also assume
		\(p_{10}p_{01}p_{1,-1}p_{-1,1}\neq0\), otherwise the consideration can be reduced to the previous cases.
	
	Here
	\[
	k_2
	=
	-p_{1,-1}p_{-1,1}
	\bigl(
	p_{01}p_{0,-1}p_{10}p_{-1,0}
	-
	p_{1,-1}^{\,2}p_{-1,1}^{\,2}
	\bigr).
	\]
	Therefore,
	\begin{equation}\label{eq:order6a}
	p_{01}p_{0,-1}p_{10}p_{-1,0}
	=
	p_{1,-1}^{\,2}p_{-1,1}^{\,2}.
	\end{equation}
Note that this condition implies \(p_{0,-1}p_{-1,0}\neq0\), so we have:
$$
p_{01}= \frac{p_{1,-1}^2 p_{-1,1}^2}{p_{0,-1} p_{10} p_{-1,0}},
$$
and we calculate:
$$
k_3=-\frac{p_{1,-1} p_{-1,1}^2 \left(p_{0,-1}^2 p_{10}^2+p_{1,-1}^3 p_{-1,1}\right) 
(p_{10} p_{-1,0}-p_{1,-1} p_{-1,1})^2}{p_{0,-1} p_{10}^2 p_{-1,0}}.
$$
Then $k_3=0$ and the fact that $p_{ij}$s are probabilities implies 
\begin{equation}\label{eq:order6b}
p_{10} p_{-1,0}-p_{1,-1} p_{-1,1}=0.
\end{equation}
	Combining \eqref{eq:order6a} and \eqref{eq:order6b} gives
	\[
	p_{10}p_{-1,0}
	=
	p_{1,-1}p_{-1,1}
	=
	p_{01}p_{0,-1}.
	\]
	Consequently,
	\[
	p_{10}=\frac{p_{1,-1}p_{-1,1}}{p_{-1,0}},
	\qquad
	p_{01}=\frac{p_{1,-1}p_{-1,1}}{p_{0,-1}},
	\]
	and a direct substitution gives \(k_4=0\). This is family \emph{(III)}.
	
	The listed cases exhaust all possibilities, completing the proof.
\end{proof}

\section{Groups of order 8}\label{sec:order8}

We now classify the probability matrices for which the associated group has
order \(8\) for every \(t\).

\begin{theorem}\label{th:order8} The weighted walks in quadrant have groups $G_W$ of order $8$ if and only if, up to the symmetries of the model, 
	\[
	p_{11}=p_{01}=p_{-1,-1}=p_{0,-1}=0,
	\]
	\[
	p_{10}p_{-1,0}=p_{1,-1}p_{-1,1},
	\]
	and
	\[
	p_{10}p_{1,-1}p_{-1,1}\neq0.
	\]
	Equivalently, the weighted walks in quadrant have groups $G_W$ of order $8$ if and only if
	\[
	P_{\infty}
	=
	\begin{pmatrix}
		0 & p_{10} & p_{1,-1}\\
		0 & p_{00} & 0\\
		p_{-1,1} & \dfrac{p_{1,-1}p_{-1,1}}{p_{10}} & 0
	\end{pmatrix},
	\qquad
	p_{10}p_{1,-1}p_{-1,1}\neq0.
	\]
\end{theorem}

\begin{proof}
According to Theorems \ref{th:finite-group} and \ref{th:cayley}, the group $G_W$ is of order $2n=6$ if and only if $C_3(t)=0$ for all $t$.

	Thus, the group $G_W$ has the order \(8\) if and only if
	\[C_3(t)=
	\frac{
		k_0+k_1t+k_2t^2+k_3t^3+k_4t^4+k_5t^5+k_6t^6
	}
	{t(\det P_{00}-t\det P_{\infty})^5}
	=0,
	\]
    for all $t$
.	Thus, this identity holds for every \(t\) precisely when the numerator
	vanishes identically and the denominator is not identically zero. Hence we
	require
	\[
	k_0=k_1=\cdots=k_6=0,
	\]
	together with the condition that at least one of \(P_{00}\) and
	\(P_{\infty}\) be non-degenerate. The latter condition excludes the case
	where the group has order \(4\) for every \(t\).
	
	We have
	\[
	k_0
	=
	p_{11}p_{1,-1}p_{-1,1}p_{-1,-1}
	\bigl(p_{11}p_{-1,-1}+p_{1,-1}p_{-1,1}\bigr).
	\]
	Since all coefficients are non-negative, \(k_0=0\) implies that at least
	one corner coefficient vanishes. By symmetry, we may assume that
	\[
	p_{11}=0.
	\]
	Then
	\[
	k_1=p_{01}p_{10}p_{1,-1}^{\,2}p_{-1,1}^{\,2}p_{-1,-1}.
	\]
	Using symmetry once more, it remains to consider the three cases
	\[
	p_{01}=0,\qquad p_{1,-1}=0,\qquad\text{or}\qquad p_{-1,-1}=0.
	\]
	
	\medskip
	\noindent\textbf{Case 1: \(p_{11}=p_{01}=0\).}
	
	In this case,
	\[
	\det P_{00}=-p_{1,-1}p_{-1,1}
\qquad\text{and}\qquad
	\det P_{\infty}
	=
	p_{-1,1}
	\det
	\begin{pmatrix}
		p_{10} & p_{1,-1}\\
		p_{00} & p_{0,-1}
	\end{pmatrix}.
	\]
	Therefore, non-degeneracy implies
	\[
	p_{-1,1}\neq0.
	\]
	Moreover,
	\[
	k_2=p_{10}^{\,2}p_{1,-1}^{\,2}p_{-1,1}^{\,3}p_{-1,-1},
	\]
so we can conclude that one of \(p_{10}\), \(p_{1,-1}\), and \(p_{-1,-1}\) must
	vanish.
	
	If \(p_{10}=0\), then
	\[
	k_4=-p_{1,-1}^{\,5}p_{-1,1}^{\,5},
	\]
	which forces \(p_{1,-1}=0\). But then both \(\det P_{00}\) and
	\(\det P_{\infty}\) vanish, a contradiction.
	
	If \(p_{1,-1}=0\), non-degeneracy gives
	\[
	p_{10}p_{0,-1}\neq0.
	\]
	Then
	\[
	k_4=p_{0,-1}^{\,2}p_{10}^{\,4}p_{-1,1}^{\,3}p_{-1,-1},
	\]
	so \(p_{-1,-1}=0\). Next,
	\[
	k_5=p_{0,-1}^{\,3}p_{10}^{\,4}p_{-1,0}p_{-1,1}^{\,3},
	\]
	forces \(p_{-1,0}=0\), but then
	\[
	k_6=p_{0,-1}^{\,4}p_{10}^{\,4}p_{-1,1}^{\,4}\neq0.
	\]
	Hence this subcase yields no solutions.
	
	It remains to consider
	\[
	p_{-1,-1}=0,
	\qquad
	p_{10}p_{1,-1}p_{-1,1}\neq0.
	\]
	Since
	\[
	k_3=
	p_{0,-1}p_{10}^{\,2}p_{1,-1}^{\,2}p_{-1,0}p_{-1,1}^{\,3},
	\]
	we have either \(p_{0,-1}=0\) or \(p_{-1,0}=0\).
	
	If \(p_{-1,0}=0\) and \(p_{0,-1}\neq0\), then
	\[
	k_4=
	p_{1,-1}^{\,2}p_{-1,1}^{\,4}
	\bigl(p_{0,-1}^{\,2}p_{10}^{\,2}-p_{1,-1}^{\,3}p_{-1,1}\bigr).
	\]
	Substituting the resulting relation into \(k_5\) gives
	\[
	k_5=-3p_{1,-1}p_{-1,1}^{\,4}p_{0,-1}^{\,3}p_{10}^{\,3}\neq0.
	\]
	Thus there are no solutions in this subcase.
	
	Therefore,
	\[
	p_{0,-1}=0.
	\]
	Now
	\[
	k_4=
	p_{1,-1}^{\,3}p_{-1,1}^{\,3}
	\bigl(p_{10}p_{-1,0}-p_{1,-1}p_{-1,1}\bigr)
	\bigl(p_{10}p_{-1,0}+p_{1,-1}p_{-1,1}\bigr).
	\]
	Since the coefficients are probabilities and
	\[
	p_{10}p_{-1,0}+p_{1,-1}p_{-1,1}>0,
	\]
	we obtain
	\[
	p_{10}p_{-1,0}=p_{1,-1}p_{-1,1}.
	\]
	The displayed factor divides both \(k_5\) and \(k_6\); hence
	\[
	k_5=k_6=0.
	\]
	This gives the family stated in the theorem.
	
	\medskip
	\noindent\textbf{Case 2: \(p_{11}=p_{1,-1}=0\).} 
	By symmetry with Case 1, we can also assume \(p_{10}p_{01}p_{0,-1}\neq0\).
	Here
	\[
	k_3=
	p_{01}p_{0,-1}p_{10}^{\,3}p_{-1,1}p_{-1,-1}
	\bigl(p_{01}p_{-1,-1}+p_{0,-1}p_{-1,1}\bigr).
	\]
	Non-negativity shows that the final factor can vanish only if
	\[
	p_{-1,1}=p_{-1,-1}=0,
	\]
	but this would imply \(\det P_{\infty}=0\). Hence either
	\(p_{-1,1}=0\) or \(p_{-1,-1}=0\), and symmetry allows us to consider only one of those possibilities.
	
	Under assumption \(p_{-1,1}=0\), non-degeneracy requires \(p_{-1,-1}\neq0\), and we have
	\[
	k_4=
	p_{01}^{\,3}p_{0,-1}p_{10}^{\,3}p_{-1,0}p_{-1,-1}^{\,2}\neq0.
	\]
	Thus Case \(2\) yields no solutions.
	
	\medskip
	\noindent\textbf{Case 3:
		\(p_{11}=p_{-1,-1}=0\).} 
	By symmetry with Cases 1 and 2, we can assume 
	\[p_{01}p_{10}p_{0,-1}p_{-1,0}p_{1,-1}p_{-1,1}\neq0.
	\]
	But then
	\[
	k_2=
	p_{01}p_{0,-1}p_{10}p_{1,-1}^{\,2}p_{-1,0}p_{-1,1}^{\,2}\neq0,
	\]
	so this case yields no solutions.
	
	The three cases exhaust all possibilities. Therefore the family displayed in
	the statement of the theorem represents the complete classification, up to symmetry.
\end{proof}

\begin{remark}
For completeness of argument in a more general case beyond probability, let us consider the case when the factor: 
\begin{equation}\label{eq:cond+}
p_{10} p_{-1,0}+p_{1,-1} p_{-1,1},
\end{equation}
which appears in Case 1 of the previous proof, is not strictly positive.
If it vanishes, we have: 
$$
k_6
=
-8p_{1,-1}^{\,3}p_{-1,1}^{\,3}
	p_{10}^{\,2}p_{-1,0}^{\,2}
	\bigl(p_{10}p_{-1,0}-p_{1,-1}p_{-1,1}\bigr),
$$
so we must have $p_{-1,0}=0$, but then the condition \eqref{eq:cond+} implies $p_{1,-1} p_{-1,1}=0$, which, according to the discussion in Case 1, will not yield any solutions.
\end{remark}

\section{Groups of order 10}

We now classify the weighted walks in quadrant with the group $G_W$ of the  
order \(10\).

\begin{theorem}\label{th:order10}
	Up to the symmetries, the weighted walks in quadrant with the group $G_W$ of the  
order \(10\), are precisely those with the probability matrices of the form:
	\[
	P_{\infty}
	=
	\frac{1}{Z}
	\begin{pmatrix}
		0 & S^3R & S^2\\
		S^3R^3 & P & 2SR\\
		S^2R^4 & 2SR^3 & R^2
	\end{pmatrix},
	\]
	where
	\[
	S>0,\qquad R>0,\qquad P\geq0,
	\]
	and
	\[
	Z
	=
	S^3R+S^2+S^3R^3+P+2SR+S^2R^4+2SR^3+R^2.
	\]
\end{theorem}

\begin{proof}
Using Theorems \ref{th:cayley} and \ref{th:finite-group}, we get that $G_W$ has order $2n=10$ if and only if $C_2(t)C_4(t)-C_3^2(t)=0$, for all $t$. This can be equivalently written as:
	\[
	\frac{k_0+k_1t+k_2t^2+\cdots+k_{12}t^{12}}
	{\bigl(\det P_{00}-t\det P_{\infty}\bigr)^{10}}
	=0,
	\]
for all $t$, with $k_0$, \dots, $k_{12}$ being polynomial in $p_{ij}$.

	Therefore, the group $G_W$ has the order \(10\) when
	\[
	k_0=k_1=\cdots=k_{12}=0
	\]
	and at least one of \(P_{00}\) and \(P_{\infty}\) is non-degenerate.
	
	Since
	\[
	k_0=p_{11}^{\,3}p_{1,-1}^{\,3}p_{-1,1}^{\,3}p_{-1,-1}^{\,3},
	\]
	at least one corner coefficient must vanish. By symmetry, we may assume
	\[
	p_{11}=0.
	\]
	Then \(k_1=k_2=0\), while
	\[
	k_3=
	p_{01}p_{10}p_{1,-1}^{\,3}p_{-1,1}^{\,3}p_{-1,-1}
	\bigl(
	p_{01}^{\,2}p_{10}^{\,2}p_{-1,-1}^{\,2}
	-
	p_{1,-1}^{\,3}p_{-1,1}^{\,3}
	\bigr).
	\]
	Thus, up to symmetry, there are four possibilities:
	\[
	p_{01}=0,\qquad
	p_{1,-1}=0,\qquad
	p_{-1,-1}=0,
	\]
	or
	\begin{equation}\label{eq:order10-balance}
		p_{01}^{\,2}p_{10}^{\,2}p_{-1,-1}^{\,2}
		=
		p_{1,-1}^{\,3}p_{-1,1}^{\,3}.
	\end{equation}
	
	\medskip
	\noindent\textbf{Cases with vanishing coefficients.}
	
	We will now show that the first three cases yield no solutions.
	
	Suppose that \(p_{01}=0\). 
	Then non-degeneracy of at least one of the matrices $P_{00}$ and $P_{\infty}$ implies
	\[
	p_{-1,1}\neq0.
	\]
	Moreover,
	\[
	k_4=-p_{10}^{\,2}p_{1,-1}^{\,6}p_{-1,1}^{\,7}p_{-1,-1}.
	\]
	The three possibilities \(p_{10}=0\), \(p_{1,-1}=0\), and
	\(p_{-1,-1}=0\) all lead to no solutions:
	
	\begin{itemize}
		\item If \(p_{10}=0\), then non-degeneracy requires \(p_{1,-1}\neq0\), whereas
		\[
		k_8=p_{1,-1}^{\,10}p_{-1,1}^{\,10}\neq0.
		\]
		
		\item If \(p_{1,-1}=0\), then non-degeneracy requires
		\(p_{0,-1}\neq0\) and we can also assume $p_{10}\neq0$. The vanishing of \(k_9\) and \(k_{11}\) successively
		forces
		\[
		p_{-1,-1}=p_{-1,0}=0,
		\]
		but then
		\[
		k_{12}=-p_{0,-1}^{\,8}p_{10}^{\,8}p_{-1,1}^{\,8}\neq0.
		\]
		
		\item If \(p_{-1,-1}=0\), we can also assume $p_{10}p_{1,-1}\neq0$. Then
		\[
		k_5=-p_{0,-1}p_{10}^{\,2}p_{1,-1}^{\,6}
		p_{-1,0}p_{-1,1}^{\,7}.
		\]
		If \(p_{0,-1}=0\), then \(k_6=0\) forces \(p_{-1,0}=0\), and hence
		\[
		k_8=p_{1,-1}^{\,10}p_{-1,1}^{\,10}\neq0.
		\]
		If \(p_{-1,0}=0\) and \(p_{0,-1}\neq0\), then
		\[
		k_6=-p_{0,-1}^{\,2}p_{10}^{\,2}p_{1,-1}^{\,6}p_{-1,1}^{\,8}\neq0.
		\]
	\end{itemize}
	
	Hence the case \(p_{01}=0\) gives no solutions. Similarly the cases
	\[
	p_{1,-1}=0
	\qquad\text{and}\qquad
	p_{-1,-1}=0
	\]
	give no solutions either. Consequently, any solution must satisfy
	\[
	p_{01}p_{10}p_{1,-1}p_{-1,1}p_{-1,-1}\neq0
	\]
	and the relation \eqref{eq:order10-balance},
	which together consititute:
	
	\medskip
	\noindent\textbf{The remaining case.}
	
	Since the coefficients \(k_3,\ldots,k_{12}\) are homogeneous, we may
	temporarily normalize
	\[
	p_{-1,-1}=1.
	\]
	Equation~\eqref{eq:order10-balance} then becomes
	\[
	p_{01}^{\,2}p_{10}^{\,2}
	=
	p_{1,-1}^{\,3}p_{-1,1}^{\,3}.
	\]
	Introduce positive parameters \(S,R,T\) by
	\[
	S^4=p_{1,-1}p_{-1,1},
	\qquad
	R^2=\frac{p_{01}}{p_{10}},
	\qquad
	T^2=\frac{p_{1,-1}}{p_{-1,1}}.
	\]
	Thus,
	\[
	p_{01}=S^3R,
	\qquad
	p_{10}=\frac{S^3}{R},
	\qquad
	p_{1,-1}=S^2T,
	\qquad
	p_{-1,1}=\frac{S^2}{T}.
	\]
	
	The equation \(k_4=0\) gives
	\[
	p_{-1,0}
	=
	-\frac{3p_{0,-1}RS-2R^4S^2T^2-2S^2}
	{R^2T(3RST-2p_{0,-1})}.
	\]
	Substituting that, the remaining equations \(k_5=\cdots=k_{10}=0\) reduce to polynomial
	equations in \(R,S,T,p_{0,-1}\). 
	Those equations can be analysed by calculating resultants with respect to variable $p_{0,-1}$.
	Namely, all resultants of $k_5$ with the remaining polynomials $k_6$, \dots, $k_{10}$ have factors		
	\[
		R^2T-1
		\qquad\text{and}\qquad
		4-9R^2T+4R^4T^2.
		\]
By analysing the resultants, one can conclude that the second factor and the remaining factors do not give any solutions.
	
On the other hand, the first factor gives:
	\[
	T=\frac1{R^2},
	\]
	and therefore
	\[
	p_{01}=S^3R,
	\qquad
	p_{10}=\frac{S^3}{R},
	\qquad
	p_{1,-1}=\frac{S^2}{R^2},
	\qquad
	p_{-1,1}=S^2R^2.
	\]
	Moreover,
	\[
	p_{-1,0}
	=
	-\frac{RS(3p_{0,-1}R-4S)}
	{3S-2p_{0,-1}R}.
	\]
	Substitution into \(k_5=0\) and \(k_6=0\) gives
	\[
	(p_{0,-1}R-2S)^2
	\bigl(
	5p_{0,-1}^{\,2}R^2
	-14p_{0,-1}RS
	+11S^2
	\bigr)=0
	\]
	and
	\[
		(p_{0,-1}R-2S)^2
		\bigl(
		31p_{0,-1}^{\,4}R^4
		-69p_{0,-1}^{\,3}R^3S
		-67p_{0,-1}^{\,2}R^2S^2
		+279p_{0,-1}RS^3
		-188S^4
		\bigr)=0.
	\]
	The resultant of the two second factors, considered as polynomials in
	\(p_{0,-1}\), is
	\[
	676S^8R^8\neq0.
	\]
	Hence
	\[
	p_{0,-1}=\frac{2S}{R},
	\qquad
	p_{-1,0}=2SR.
	\]
	
	Finally, undoing the normalization \(p_{-1,-1}=1\) by multiplying the
	matrix by \(R^2\), and then normalizing its entries to have total sum one,
	gives
	\[
	\frac{1}{Z}
	\begin{pmatrix}
		0 & S^3R & S^2\\
		S^3R^3 & P & 2SR\\
		S^2R^4 & 2SR^3 & R^2
	\end{pmatrix}.
	\]
	Here \(P\geq0\) is arbitrary, since \(p_{00}\) does not occur in the
	preceding relations. Direct substitution shows that all coefficients
	\(k_0,\ldots,k_{12}\) vanish for this family. This completes the proof.
\end{proof}

\section{Groups of order 12}

We show that there are no weighted walks whose associated group $G_W$ has the order
\(12\).

\begin{theorem}\label{th:order12}
	There are no random walks in quadrant, for which the associated
	group $G_W$  has the order \(12\).
\end{theorem}

\begin{proof}
Theorems \ref{th:cayley} and \ref{th:finite-group} give that $G_W$ is of order $2n=12$ if and only if $C_3(t)C_5(t)-C_4^2(t)=0$, for all $t$, while $C_2(t)C_3(t)\neq0$ and, at least one of \(P_{00}\) and \(P_{\infty}\) is non-degenerate.

The condition $C_3(t)C_5(t)-C_4(t)^2=0$ for all $t$ is equivalent to
	\[
	k_0+k_1t+k_2t^2+\cdots+k_{12}t^{12}=0,
	\]
for all $t$, where $k_0$, \dots, $k_{12}$ are polynomial in $p_{ij}$.
	Thus, for the group $G_W$ to be of order $12$ it is necessary to have
	\[
	k_0=k_1=\cdots=k_{12}=0.
	\]
	
	We calculate:
	\[
	k_0
	=
	-p_{11}^{\,2}p_{1,-1}^{\,2}p_{-1,1}^{\,2}p_{-1,-1}^{\,2}
	\left(
	p_{11}^{\,2}p_{-1,-1}^{\,2}
	+p_{11}p_{1,-1}p_{-1,1}p_{-1,-1}
	+p_{1,-1}^{\,2}p_{-1,1}^{\,2}
	\right).
	\]
	Since the coefficients are non-negative, \(k_0=0\) implies that at least one
	corner coefficient vanishes. By symmetry, we may assume
	\[
	p_{11}=0.
	\]
	Then \(k_1=0\), and
	\[
	k_2=-p_{01}^{\,2}p_{10}^{\,2}p_{1,-1}^{\,4}
	p_{-1,1}^{\,4}p_{-1,-1}^{\,2}.
	\]
	Hence, up to symmetry, we need only consider the cases
	\[
	p_{01}=0,
	\qquad
	p_{1,-1}=0,
	\qquad\text{or}\qquad
	p_{-1,-1}=0.
	\]
	
	\medskip
	\noindent\textbf{Case 1: \(p_{11}=p_{01}=0\).}
	
	Non-degeneracy implies
	\[
	p_{-1,1}\neq0.
	\]
	Moreover,
	\[
	k_4=
	-p_{10}^{\,2}p_{1,-1}^{\,4}p_{-1,1}^{\,6}p_{-1,-1}
	\left(
	p_{10}^{\,2}p_{-1,-1}
	+p_{1,-1}^{\,2}p_{-1,1}
	\right).
	\]
	Therefore, one of \(p_{10}\), \(p_{1,-1}\), and \(p_{-1,-1}\) must vanish.
	
	If \(p_{10}=0\), non-degeneracy requires \(p_{1,-1}\neq0\), whereas
	\[
	k_{11}=p_{1,-1}^{\,11}p_{-1,1}^{\,11}\neq0.
	\]
	Thus this subcase is impossible.
	
	If \(p_{1,-1}=0\), then non-degeneracy implies
	\(
	p_{0,-1}\neq0.
	\), and we can also assume $p_{10}\neq0$.
	The equation
	\[
	k_8=-p_{0,-1}^{\,4}p_{10}^{\,8}p_{-1,1}^{\,6}p_{-1,-1}^{\,2}
	\]
	forces \(p_{-1,-1}=0\). Then
	\[
	k_{10}
	=
	-p_{0,-1}^{\,6}p_{10}^{\,8}p_{-1,0}^{\,2}p_{-1,1}^{\,6}
	\]
	forces \(p_{-1,0}=0\), but
	\[
	k_{12}
	=
	-2p_{0,-1}^{\,8}p_{10}^{\,8}p_{-1,1}^{\,8}\neq0.
	\]
	Hence this subcase is impossible.
	
	It remains to consider
	\(
	p_{-1,-1}=0\), with
	\(p_{10}p_{1,-1}p_{-1,1}\neq0.
	\)
	We have
	\[
	k_5=
	-p_{0,-1}p_{10}^{\,2}p_{1,-1}^{\,6}p_{-1,0}p_{-1,1}^{\,7}.
	\]
	Thus either \(p_{0,-1}=0\) or \(p_{-1,0}=0\).
	
	If \(p_{0,-1}=0\), then
	\[
	k_6=
	-p_{10}^{\,2}p_{1,-1}^{\,7}p_{-1,0}^{\,2}p_{-1,1}^{\,7},
	\]
	so \(p_{-1,0}=0\). But then
	\[
	k_{10}=p_{1,-1}^{\,11}p_{-1,1}^{\,11}\neq0.
	\]
	
	If \(p_{-1,0}=0\) and \(p_{0,-1}\neq0\), then
	\[
	k_6=
	-p_{0,-1}^{\,2}p_{10}^{\,2}p_{1,-1}^{\,6}p_{-1,1}^{\,8}\neq0.
	\]
	Thus Case \(1\) yields no solutions.
	
	\medskip
	\noindent\textbf{Case 2:
		\(p_{11}=p_{1,-1}=0\).} 
By symmetry with Case 1, we can assume \(p_{01}p_{10}\neq0\).
	
	We have
	\[
		k_6=
		-p_{01}^{\,2}p_{0,-1}^{\,2}p_{10}^{\,6}p_{-1,1}^{\,2}
		p_{-1,-1}^{\,2}
		\left(
		p_{01}^{\,2}p_{-1,-1}^{\,2}
		+p_{01}p_{0,-1}p_{-1,1}p_{-1,-1}
		+p_{0,-1}^{\,2}p_{-1,1}^{\,2}
		\right).
	\]
	Hence one of
	\[
	p_{0,-1}=0,
	\qquad
	p_{-1,1}=0,
	\qquad\text{or}\qquad
	p_{-1,-1}=0
	\]
	must hold.
	
	If \(p_{0,-1}=0\), non-degeneracy gives \(p_{-1,-1}\neq0\). Then
	\[
	k_8=
	-p_{01}^{\,4}p_{10}^{\,8}p_{-1,1}^{\,2}p_{-1,-1}^{\,6}
	\]
	forces \(p_{-1,1}=0\), and
	\[
	k_{10}=
	-p_{01}^{\,6}p_{10}^{\,8}p_{-1,0}^{\,2}p_{-1,-1}^{\,6}
	\]
	forces \(p_{-1,0}=0\). But then
	\[
	k_{12}=
	-2p_{01}^{\,8}p_{10}^{\,8}p_{-1,-1}^{\,8}\neq0.
	\]
	
	If \(p_{-1,1}=0\), non-degeneracy gives \(p_{-1,-1}\neq0\) and we can also assume $p_{0,-1}\neq0$. Then
	\[
	k_8=
	-p_{01}^{\,6}p_{0,-1}^{\,2}p_{10}^{\,6}
	p_{-1,0}^{\,2}p_{-1,-1}^{\,4}
	\]
	forces \(p_{-1,0}=0\), but
	\[
	k_{10}=
	-p_{01}^{\,8}p_{0,-1}^{\,2}p_{10}^{\,6}p_{-1,-1}^{\,6}\neq0.
	\]
	
	Finally, if \(p_{-1,-1}=0\), then we can also assume $p_{0,-1}p_{-1,1}\neq0$ and calculate
	\[
	k_8=
	-p_{01}^{\,2}p_{0,-1}^{\,6}p_{10}^{\,6}
	p_{-1,0}^{\,2}p_{-1,1}^{\,4},
	\]
	which forces \(p_{-1,0}=0\), while
	\[
	k_{10}=
	-p_{01}^{\,2}p_{0,-1}^{\,8}p_{10}^{\,6}p_{-1,1}^{\,6}\neq0.
	\]
	Therefore Case \(2\) gives no solutions.
	
	\medskip
	\noindent\textbf{Case 3:
		\(p_{11}=p_{-1,-1}=0\).}
By symmetry with Cases 1 and 2, we can also assume 
\[p_{01}p_{10}p_{1,-1}p_{-1,1}p_{0,-1}p_{-1,0}\neq0,\]
but then
	\[
		k_4=
		-p_{01}p_{0,-1}p_{10}p_{1,-1}^{\,4}p_{-1,0}p_{-1,1}^{\,4}
		\left(
		p_{01}p_{0,-1}p_{10}p_{-1,0}
		+p_{1,-1}^{\,2}p_{-1,1}^{\,2}
		\right)\neq0.
	\]
	Hence, this case gives no solutions.
	
	The three cases exhaust all the possibilities and we conclude that there are no weighted walks
	whose associated group has order \(12\).
\end{proof}

\section{Conclusion}

We summarize our previous considerations.

\begin{theorem}\label{th:conclusion} All weighted walks in quadrant with a finite group $G_W$ are described in Theorems \ref{th:order4}, \ref{th:order6},  \ref{th:order8}, and \ref{th:order10}. There are no weighted walks in quadrant with groups $G_W$ of order higher than $10$.
\end{theorem}

\begin{proof} The proof follows from the proofs of Theorems \ref{th:order4}, \ref{th:order6},  \ref{th:order8}, and \ref{th:order10}, as well as the proofs of Theorems \ref{th:uniformbound} and \ref{th:order12}.
\end{proof}

Let us observe that the method of the proof of Theorem \ref{th:order12} can be used to show that there is no weighted walk in quadrant with the group $G_W$ of any given order $2n$, for $n>6$.

The results of Theorems \ref{th:order4}, \ref{th:order6},  \ref{th:order8}, and \ref{th:order10}, as well as the proofs of Theorems \ref{th:uniformbound} and \ref{th:order12}, as summarized in Theorem \ref{th:conclusion} coincide with the results formulated in \cite{KaYa} for orders $4, 6, 8$ and in \cite{EHR} for orders $10$ and higher, and are logically independent from them.

\

\subsection*{Acknowledgments}
We are  grateful  to Andrew Elvey Price, Emmanuel Humbert, and Kilian Raschel for sharing their interesting preprint \cite{EHR} with us and for a fruitful discussion that led to Theorem \ref{th:finite-group}. V.D. also thanks them for hospitality during his visit to Tours. 
This research was partially supported by Simons Foundation grant no.~854861.

\end{document}